\documentclass[11pt, letterpaper]{amsart}
\usepackage[utf8]{inputenc}
\usepackage{setspace}
\usepackage{mathrsfs}
\usepackage{amssymb}
\usepackage{latexsym}
\usepackage{amsfonts}
\usepackage{amsmath}
\usepackage{eucal}
\usepackage{bm}
\usepackage{bbm}
\usepackage{graphicx}
\usepackage[english]{varioref}
\usepackage[nice]{nicefrac}
\usepackage[all]{xy}
\usepackage{amsthm}
\usepackage{amssymb,amsthm,upref,amscd}

\def\op{\operatorname}

\def\mmod{\kern-1pt\operatorname{-mod}}

\newtheorem{theorem}{Theorem}[section]
\newtheorem{lemma}[theorem]{Lemma}
\newtheorem{definition}[theorem]{Definition}

\newtheorem{remark}[theorem]{Remark}
\newtheorem{proposition}[theorem]{Proposition}

\newtheorem{conjecture}[theorem]{Conjecture}
\newtheorem{question}[theorem]{Question}
\theoremstyle{proposition}

\numberwithin{equation}{section}

\begin{document}

\title[Tensor product in the category $\mathscr{X}$]{Some conjectures on the quotients of the tensor products   in the category $\mathscr{X}$ }

%    Information for first author
\author{Junbin Dong}
%    Address of record for the research reported here
\address{Institute of Mathematical Sciences, ShanghaiTech University, 393 Middle Huaxia Road, Pudong, Shanghai 201210, China.}
%    Current address
\email{dongjunbin@shanghaitech.edu.cn}
%    \thanks will become a 1st page footnote.
%\thanks{}

%    General info

% \date{September 23, 2023}


\begin{abstract} Let ${\bf G}$ be a connected reductive algebraic group defined over the finite field $\mathbb{F}_q$ with $q$ elements.
We propose some conjectures concerning the simple quotients of $M\otimes N$, where $M,N$  are objects in  the representation category $\mathscr{X}({\bf G})$ introduced by the author in a previous work to study the complex representations of  ${\bf G}$. We provide several pieces of evidence for these conjectures. In particular, we show that  these conjectures are valid for  ${\bf G}=SL_2(\bar{\mathbb{F}}_q)$.
\end{abstract}

\maketitle

\section{Introduction}

Let ${\bf G}$ be a connected reductive algebraic  group defined over the finite field $\mathbb{F}_q$ with $q$ elements.  Let $\Bbbk$ be an algebraically closed field of characteristic zero,  and all the representations considered  in this paper are over $\Bbbk$.
Let $\bf T$ be a maximal  torus contained in a  Borel subgroup $\bf B$ and $\theta$  be   a character of ${\bf T}$. Thus $\theta$  can also be regarded as a character of ${\bf B}$ by letting ${\bf U}$ (the unipotent radical of ${\bf B}$) act trivially.  The naive induced module $\mathbb{M}(\theta)=\Bbbk{\bf G}\otimes_{\Bbbk{\bf B}}{\Bbbk}_\theta$  has been studied in \cite{CD3}. The  induced module $\mathbb{M}(\theta)$  has a composition series of finite length,  and its composition factors are $E(\theta)_J$ with $J\subset I(\theta)$ (see Section 2 for the explicit definition of $I(\theta)$ and $E(\theta)_J$). Consequently,  we have obtained a large class of irreducible modules of ${\bf G}$, almost all of which are infinite dimensional. Not long after that, J.B. Dong introduced the principal representation category $\mathscr{O}({\bf G})$  in  \cite{D1}, which  was conjectured to be a highest weight category in the sense of \cite{CPS}.  In particular,  extensions of the simple modules in $\mathscr{O}({\bf G})$ were expected to  have good properties.  However, shortly thereafter,   X.Y. Chen constructed a counterexample (see \cite{C3}) showing  that this conjecture is not true. His example also indicates that the  category $\mathscr{O}({\bf G})$ may be more complicated than we thought even for ${\bf G}= SL_2(\overline{\mathbb{F}}_q)$.

In the paper \cite{D3}, we introduce the representation category $\mathscr{X}({\bf G})$, whose definition is inspired by the structure of  the $\Bbbk {\bf G}$-modules $\mathbb{M}(\theta)_J$ and  $E(\theta)_J$.  We show that the category $\mathscr{X}({\bf G})$ is an abelian category and also a highest weight category in the sense of \cite{CPS}. We also classify all the simple objects in the category $\mathscr{X}({\bf G})$. For any two $\Bbbk {\bf G}$-modules $M,N \in \mathscr{X}({\bf G})$, the tensor product $M\otimes N$ does not generally lie in $\mathscr{X}({\bf G})$. From many examples, we observe that the submodules of $M\otimes N$  are not in $\mathscr{X}({\bf G})$ either. However,  $M\otimes N$ has many simple quotients that belong to $\mathscr{X}({\bf G})$. In this paper, we put forward some conjectures concerning  the simple  quotients of $M\otimes N$ for any two objects   $M, N\in \mathscr{X}({\bf G})$.

The rest of this paper is organized as follows: In Section 2, we present some notations and preliminary results. In Section 3, we propose several conjectures concerning the simple quotients  of $M\otimes N$  for any two objects $M,N \in \mathscr{X}({\bf G})$,  and provide some evidence to support them. In Section 4, we show that these conjectures hold for ${\bf G}= SL_2(\overline{\mathbb{F}}_q)$.

\section{Notations and Preliminaries}

As in the introduction,  ${\bf G}$ is a connected reductive algebraic group defined over $\mathbb{F}_q$ with the standard Frobenius homomorphism $\text{Fr}$ induced by the automorphism $x\mapsto x^q$ on $\bar{\mathbb{F}}_q$, where $q$ is a power of a prime number $p$.   Let ${\bf B}$ be an $\text{Fr}$-stable Borel subgroup, and ${\bf T}$ be an $\text{Fr}$-stable maximal torus contained in ${\bf B}$, and ${\bf U}=R_u({\bf B})$ be the ($\text{Fr}$-stable) unipotent radical of ${\bf B}$. We identify ${\bf G}$ with ${\bf G}(\bar{\mathbb{F}}_q)$ and do likewise for the various subgroups of ${\bf G}$ such as ${\bf B}, {\bf T}, {\bf U}$ $\cdots$. We denote by $\Phi=\Phi({\bf G};{\bf T})$ the corresponding root system, and by $\Phi^+$ (resp. $\Phi^-$) the set of positive (resp. negative) roots determined by ${\bf B}$. Let $W=N_{\bf G}({\bf T})/{\bf T}$ be the corresponding Weyl group and denote by $\ell(w)$ the length of $w\in W$. We denote by $\Delta=\{\alpha_i\mid i\in I\}$ the set of simple roots, and by $S=\{s_i\mid i\in I\}$ the corresponding simple reflections. For each $\alpha\in\Phi$, let ${\bf U}_\alpha$ be the root subgroup corresponding to $\alpha$ and we fix an isomorphism $\varepsilon_\alpha: \bar{\mathbb{F}}_q\rightarrow{\bf U}_\alpha$ such that $t\varepsilon_\alpha(c)t^{-1}=\varepsilon_\alpha(\alpha(t)c)$ for any $t\in{\bf T}$ and $c\in\bar{\mathbb{F}}_q$. For any $w\in W$, let $\Phi_w^+=\{\alpha\in\Phi^+\mid w(\alpha)\in\Phi^+\}$ and $\Phi_w^-=\{\alpha\in\Phi^+\mid w(\alpha)\in\Phi^-\}$.
Let ${\bf U}_w$ (resp.  ${\bf U}'_w$ ) be the subgroup of ${\bf U}$ generated by all ${\bf U}_\alpha$  with $\alpha\in\Phi_w^-$  (resp.  $\alpha\in\Phi_w^+$ ).  One can refer to \cite{Car} for  the structure theory of algebraic groups.

Let $\Bbbk$ be an algebraically closed field of characteristic zero.  In this paper, we consider the abstract representations of  ${\bf G}$ over $\Bbbk$. For any finite subset $X$ of ${\bf G}$, let $\underline{X}:=\displaystyle \sum_{x\in X}x \in \Bbbk {\bf G}$.  For each positive integer $n$, we denote $G_{q^n}$ for the $\mathbb{F}_{q^n}$-points of $\bf G$, and do likewise for $B_{q^n},T_{q^n},U_{q^n}$ etc.

 Let $\widehat{\bf T}$ be the character group of ${\bf T}$.  The Weyl group $W$ acts on $\widehat{\bf T}$ by
$$(w\cdot \theta ) (t):=\theta^w(t)=\theta(\dot{w}^{-1}t\dot{w}), \quad \forall \theta\in \widehat{\bf T},$$
where $\dot{w}$  denotes an arbitrary representative of $w \in W$.
Each $\theta\in\widehat{\bf T}$ can be regarded as a character of ${\bf B}$ by letting ${\bf U}$ act trivially.
Denote by  ${\Bbbk}_\theta$ the corresponding ${\bf B}$-module and we consider the induced module $\mathbb{M}(\theta)=\Bbbk{\bf G}\otimes_{\Bbbk{\bf B}}{\Bbbk}_\theta$. Let ${\bf 1}_{\theta}$ be a fixed nonzero element in ${\Bbbk}_\theta$.  For $x\in {\bf G}$, we abbreviate $x{\bf 1}_{\theta}:=x\otimes{\bf 1}_{\theta}\in\mathbb{M}(\theta)$ for simplicity. It is not difficult to see that $\mathbb{M}(\theta)$ has a basis $\{u \dot{w} {\bf 1}_{\theta}\mid w\in W,  u\in {\bf U}_{w^{-1}}\}$ by the Bruhat decomposition, where $\dot{w}$ is a fixed representative of $w \in W$.

For each $i \in I$, let ${\bf G}_i$ be the subgroup of $\bf G$ generated by ${\bf U}_{\alpha_i}, {\bf U}_{-\alpha_i}$ and we set ${\bf T}_i= {\bf T}\cap {\bf G}_i$. For $\theta\in\widehat{\bf T}$, we let $$I(\theta)=\{i\in I \mid \theta| _{{\bf T}_i} \ \text {is trivial}\}.$$
For $J\subset I(\theta)$, let $W_J$ be the subgroup of $W$ generated by $s_i~(i\in J)$ and $w_J$ be the longest element of $W_J$.  Let ${\bf G}_J$ be the subgroup of $\bf G$ generated by ${\bf G}_i~(i\in J)$. We choose a representative $\dot{w}\in {\bf G}_J$ for each $w\in W_J$. Thus, the element $w{\bf 1}_\theta:=\dot{w}{\bf 1}_\theta$  $(w\in W_J)$ is well defined.
For $J\subset I(\theta)$, we set
$$\eta(\theta)_J=\sum_{w\in W_J}(-1)^{\ell(w)}w{\bf 1}_{\theta},$$
where $\ell(w)$ is the length of  $w\in W$.  Let $\mathbb{M}(\theta)_J=\Bbbk{\bf G}\eta(\theta)_J$ be the $\Bbbk {\bf G}$-module which is generated by $\eta(\theta)_J$. For $w\in W$, let  $\mathscr{R}(w)=\{i\in I\mid ws_i<w\}$.  For any subset $J\subset I$, we set
$$
\aligned
W^J &\ =\{x\in W\mid x~\op{has~minimal~length~in}~xW_J\}.
\endaligned
$$
Combining \cite[Proposition 2.5]{CD3} and \cite[(2.10)]{CD3}, we have the following proposition.

\begin{proposition}\label{MJ=KUW}
For any $J\subset I(\theta)$, the $\Bbbk {\bf G}$-module $\mathbb{M}(\theta)_J$ has the form
\begin{align} \label{MJ}
\mathbb{M}(\theta)_J=\sum_{w\in W^J}\Bbbk{\bf U}\dot{w}\eta(\theta)_J=\sum_{w\in W^J}\Bbbk{\bf U}_{w_Jw^{-1}}\dot{w}\eta(\theta)_J,
\end{align}
and the set $\{u\dot{w}\eta(\theta)_J \mid w\in W^J, u\in {\bf U}_{w_Jw^{-1}} \}$ is a basis of $\mathbb{M}(\theta)_J$.
\end{proposition}

For $J\subset I(\theta)$, we define
$$E(\theta)_J=\mathbb{M}(\theta)_J/\mathbb{M}(\theta)_J',$$
where $\mathbb{M}(\theta)_J'$ is the sum of all $\mathbb{M}(\theta)_K$ with $J\subsetneq K\subset I(\theta)$. We denote by $C(\theta)_J$ the image of $\eta(\theta)_J$ in $E(\theta)_J$.
Set
$$
\aligned
Z_J(\theta)&\ =\{w\in W^J \mid \mathscr{R}(ww_J)\subset J\cup (I\backslash I(\theta))\}.
\endaligned
$$

\begin{proposition} \cite[Proposition 2.7]{CD3} \label{DesEJ}
For $J\subset I(\theta)$, we have
\begin{align} \label{EJ}
E(\theta)_J=\sum_{w\in Z_J(\theta)}\Bbbk {\bf U}_{w_Jw^{-1}}\dot{w}C(\theta)_J,
\end{align}
and  the set $\{u\dot{w}C(\theta)_J \mid w\in Z_J(\theta), u\in {\bf U}_{w_Jw^{-1}} \}$ is a basis of $E(\theta)_J$.
\end{proposition}

By  \cite[Theorem 3.1]{CD3},  the composition factors of  $\mathbb{M}(\theta)$  are $E(\theta)_J$ $ (J\subset I(\theta))$,  with each of multiplicity one. According to \cite[Proposition 2.8]{CD3}, one has that  $E(\theta_1)_{K_1}$ is isomorphic to $E(\theta_2)_{K_2}$ as $\Bbbk {\bf G}$-modules if and only if $\theta_1=\theta_2$ and $K_1=K_2$. A special case is that $\theta=\op{tr}$ and $J=I$. Thus we get the Steinberg module $\op{St}$, which is generated by $\displaystyle \eta=\sum_{w\in W} (-1)^{\ell(w)} w {\bf 1}_{\op{tr}}$.  By Proposition \ref{MJ=KUW}, we have $\op{St}= \Bbbk {\bf U} \eta$.

Inspired by the property of $\mathbb{M}(\theta)_J$, $E(\theta)_J$, we construed a representation category $\mathscr{X}({\bf G})$ in \cite{D3}. For the convenience of later discussion,   we recall the definition of the category $\mathscr{X}({\bf G})$.
Let $M$ be a  $\Bbbk {\bf G}$-module. For  a character $\theta\in \widehat{\bf T}$,  let
$$ M_{\theta}= \{v\in M \mid  t v= \theta(t) v,\  \forall \ t\in {\bf T} \}$$
and we set $M_{\bf T}=\displaystyle  \bigoplus_{\theta\in \widehat{\bf T} } M_{\theta} $. Denote
$\mathbb{X}_{\bf T}(M)=\{\theta\in \widehat{\bf T} \mid   M_{\theta} \ne 0\}.$
Combining  \cite[Definition 4.1]{D3} and  \cite[Proposition 4.12]{D3}, we get the definition of $\mathscr{X}({\bf G})$ as follows:

\begin{definition}
The category $\mathscr{X}({\bf G})$ is defined to be the full subcategory of $\Bbbk {\bf G}$-Mod whose objects are the modules
 satisfying the following  condition:

 \noindent ($\star$)  $M_{\bf T}$ is a finite-dimensional space and there exists a basis $\{\xi_1, \xi_2, \dots, \xi_m\}$ of   $M_{\bf T}$, which consists of ${\bf T}$-eigenvectors such that $M=\displaystyle \bigoplus_{i=1}^m \Bbbk {\bf U} \xi_i$, and moreover for each $i=1,2,\dots, m$,  there exists $w_i\in W$ such that $\Bbbk {\bf U} \xi_i \cong \op{Ind}^{\bf U}_{{\bf U}'_{w_i}}\op{tr}$ as $\Bbbk {\bf U}$-modules.
\end{definition}

According to \cite[Theorem 5.6]{D3},  $\mathscr{X}({\bf G})$ is an abelian category. Moreover, we have classified all the simple objects in $\mathscr{X}({\bf G})$. They are exactly the $\Bbbk {\bf G}$-modules $E(\theta)_J$,  where  $\theta\in \widehat{\bf T}$ and  $J\subset I(\theta)$ (see \cite[Theorem 5.10]{D3}).

\section{Quotients of the tensor products in $\mathscr{X}({\bf G})$}

For any $M, N\in \mathscr{X}({\bf G})$, we see that $M\otimes N$ is not in $\mathscr{X}({\bf G})$ in general.
The case of ${\bf G}=SL_2(\bar{\mathbb{F}}_q)$ (see Remark \ref{remark}) also shows that the simple quotients of  $M\otimes N$  are not in $\mathscr{X}({\bf G})$ in general. To study the simple quotients of $M, N\in \mathscr{X}({\bf G})$, we put forward two conjectures.

\begin{conjecture}\label{Homfinite}
Let $M,N \in  \mathscr{X}({\bf G})$. One has that
$$\dim_{\Bbbk} \op{Hom}_{\Bbbk \bf G}(M\otimes N, L)< \infty$$ for any simple object $L \in  \mathscr{X}({\bf G})$.
\end{conjecture}

\begin{conjecture}\label{Homzero}
Let $M,N \in  \mathscr{X}({\bf G})$. For any simple object $L \in  \mathscr{X}({\bf G})$, we have
$\op{Hom}_{\Bbbk \bf G}(M\otimes N,  L)=0$  if  $\mathbb{X}_{\bf T}(L) \cap \mathbb{X}_{\bf T}(M\otimes N)= \emptyset$.
\end{conjecture}

If these two conjectures hold, then we can define the maximal semisimple quotient of $M\otimes N$ in $\mathscr{X}({\bf G})$.
Denote by $\mathcal{Q}(M\otimes N)$   the maximal semisimple quotient of $M\otimes N$ in $\mathscr{X}({\bf G})$.
For any fixed $M,N$, we set $r_{\theta, J}=\dim \op{Hom}_{\Bbbk \bf G}(M\otimes N,  E(\theta)_J)$. Then it is easy to see that
$$\mathcal{Q}(M\otimes N) \cong \bigoplus_{\theta, J}  E(\theta)^{ \oplus  r_{\theta, J} }_J.$$

\begin{proposition} \label{tensorinducedmodule}
For any  $\lambda, \mu \in \widehat{\bf T}$ and $J\subset I(\lambda)$,  we have
 $$\mathbb{M}(\lambda)_J  \otimes\mathbb{M}(\mu)   \cong   \displaystyle \bigoplus_{w\in W^J}\op{Ind}^{\bf G}_{{\bf B}_w}  \Bbbk_{\lambda^w \mu},$$ where ${\bf B}_w= {\bf T} \ltimes {\bf U}'_{w_Jw^{-1}} $. In particular, we have  $ \op{St}  \otimes \mathbb{M}(\theta)  \cong \op{Ind}^{\bf G}_{\bf T}  \Bbbk_{\theta}$ for any  $\theta\in \widehat{\bf T}$.
\end{proposition}

\begin{proof} Using \cite[Lemma 1.3]{Xi}, we have
$$\mathbb{M}(\lambda)_J  \otimes\mathbb{M}(\mu)   \cong  \op{Ind}^{\bf G}_{\bf B} (\op{Res}^{\bf G}_{\bf B} \mathbb{M}(\lambda)_J \otimes \Bbbk_{\mu})$$
as $\Bbbk {\bf G}$-modules. By Proposition \ref{MJ=KUW}, it is easy to see that
$$\mathbb{M}(\lambda)_J=\displaystyle\bigoplus_{w\in W^J}\Bbbk{\bf U}_{w_Jw^{-1}}\dot{w}\eta(\lambda)_J$$
as $\Bbbk {\bf B}$-modules. Thus we get
 $$\mathbb{M}(\lambda)_J  \otimes\mathbb{M}(\mu)   \cong  \bigoplus_{w\in W^J} \op{Ind}^{\bf G}_{\bf B} (
\Bbbk{\bf U}_{w_Jw^{-1}}\dot{w}\eta(\lambda)_J  \otimes \Bbbk_{\mu}).$$
 For any $w\in W^J$,  let  $M_{w}=  \op{Ind}^{\bf G}_{\bf B} (
\Bbbk{\bf U}_{w_Jw^{-1}}\dot{w}\eta(\lambda)_J  \otimes \Bbbk_{\mu}).$  It is easy to see that $M_w$ is generated by $\dot{w}\eta(\lambda)_J  \otimes {\bf 1}_{\mu}$. We show that $ \displaystyle \op{Ind}^{\bf G}_{{\bf B}_w}  \Bbbk_{\lambda^w \mu} \cong M_w$ as $\Bbbk {\bf G}$-modules.  Let
$$\psi:  \op{Ind}^{\bf G}_{{\bf B}_w}  \Bbbk_{\lambda^w \mu}  \rightarrow   M_w, \quad   g{\bf 1}_{\lambda^w \mu} \rightarrow g (\dot{w}\eta(\lambda)_J \otimes {\bf 1}_{\mu}).$$
It is easy to check that $\psi$ is a homomorphism of $\Bbbk {\bf G}$-modules, which is surjective.
For $g=u\dot{z}b\in {\bf G}$,  where $u\in {\bf U}_{z},  z \in W, b\in {\bf B}$, we have
$$ u\dot{z}b (\dot{w}\eta(\lambda)_J\otimes {\bf 1}_{\mu})= \mu(b) (u\dot{z}b\dot{w}\eta(\lambda)_J \otimes  u\dot{z} {\bf 1}_{\mu}).$$
Note the form of the elements on the right-hand side of the tensor product. To determine  $\text{Ker}(\psi)$,  we just need to consider $\sum k_i u \dot{z}b_i  {\bf 1}_{\lambda^w \mu} \in  \text{Ker}(\psi)$,  where $k_i\in {\Bbbk}, u\in {\bf U}_{z},  z \in W,  b_i\in {\bf U}_{w_Jw^{-1}} $. We see that
$$\psi(\sum k_i u\dot{z}b_i{\bf 1}_{\lambda^w \mu})=u\dot{z}\sum k_i b_i(\dot{w}\eta(\lambda)_J \otimes {\bf 1}_{\mu}) =0.$$
Thus $\sum k_i (b_i\dot{w}\eta(\lambda)_J \otimes {\bf 1}_{\mu}) =0. $
So we get all $k_i=0$ and  $\psi$ is an isomorphism of $\Bbbk {\bf G}$-modules. The proposition is proved.
\end{proof}

\begin{proposition} \label{conjholdforinduced}
Conjecture  \ref{Homfinite} and  Conjecture  \ref{Homzero} hold for $M=\mathbb{M}(\lambda)_J $ and $N= \mathbb{M}(\mu) $, where $\lambda, \mu \in \widehat{\bf T}$.
\end{proposition}

\begin{proof}
By Proposition \ref{tensorinducedmodule}, we have
$$\op{Hom}_{\Bbbk \bf G}(\mathbb{M}(\lambda)_J  \otimes\mathbb{M}(\mu),  E(\theta)_J) \cong \bigoplus_{w\in W^J} \op{Hom}_{\Bbbk \bf G}(\op{Ind}^{\bf G}_{{\bf B}_w}  \Bbbk_{\lambda^w \mu},  E(\theta)_J)$$
for any simple object $E(\theta)_J$.
Using Frobenius reciprocity,  for any $w\in W^J$, we get
$$\op{Hom}_{\Bbbk \bf G}(\op{Ind}^{\bf G}_{{\bf B}_w}  \Bbbk_{\lambda^w \mu},  E(\theta)_J) \cong \op{Hom}_{\Bbbk {\bf B}_w}(\Bbbk_{\lambda^w \mu},  E(\theta)_J),$$
which is a subspace of $\op{Hom}_{\Bbbk {\bf T}}(\Bbbk_{\lambda^w \mu},  E(\theta)_J)$. Noting that
$$\dim \op{Hom}_{\Bbbk {\bf T}}(\Bbbk_{\lambda^w \mu},  E(\theta)_J) < \dim (E(\theta)_J)^{\bf T} \leq |W|,$$
therefore Conjecture \ref{Homfinite} holds.  If  $\mathbb{X}_{\bf T}(E(\theta)_J) \cap \mathbb{X}_{\bf T}(\mathbb{M}(\lambda)_J  \otimes\mathbb{M}(\mu))= \emptyset$, then   $\theta \notin W\cdot \lambda^w \mu$ for any $w\in W$.
In this case,   we  have
$ \op{Hom}_{\Bbbk {\bf T}}(\Bbbk_{\lambda^w \mu},  E(\theta)_J) =0$ for any $w\in W$.  Thus Conjecture \ref{Homzero} holds. The proposition  is proved.
\end{proof}

In fact, Conjecture \ref{Homzero} can be simplified to  the following conjecture.

\begin{conjecture} \label{specialconjecture}
 One has that $\op{Hom}_{\Bbbk \bf G}(\op{St}\otimes \op{St},  E(\theta)_J)=0$ for any nontrivial character $\theta\in\widehat{ {\bf T}}$ and $J\subset I(\theta)$.
\end{conjecture}

Note that Conjecture \ref{specialconjecture} is a special case of   Conjecture \ref{Homzero}. Now we give the proof of Conjecture  \ref{Homzero} assuming Conjecture \ref{specialconjecture} is valid. We have a simple observation which is useful for later discussion. For any $M\in \mathscr{X}({\bf G}) $ with the form $M=\displaystyle \bigoplus_{i=1}^m \Bbbk {\bf U} \xi_i $,  we see that  $\displaystyle \bigoplus_{i=1}^m \Bbbk {\bf U}_{w_J} \xi_i$ is a $\Bbbk {\bf G}_{J}$-module by \cite[Proposition 4.11]{D3}. We denote by $\mathfrak{P}_J(M)= \displaystyle \bigoplus_{i=1}^m \Bbbk {\bf U}_{w_J} \xi_i \in  \mathscr{X}({\bf G}_J)$.

\begin{proof}[Proof of (Conjecture \ref{specialconjecture} $\Rightarrow$ Conjecture \ref{Homzero}). ]  In this proof, we assume that Conjecture \ref{specialconjecture} is valid.  Note that we have got all the simple objects in $\mathscr{X}({\bf G})$. Therefore we just need to  show that $$\op{Hom}_{\Bbbk \bf G}(E(\lambda)_K \otimes E(\mu)_L,  E(\theta)_J)= 0,$$
 where $\mathbb{X}_{\bf T}(E(\theta)_J) \cap  \mathbb{X}_{\bf T}(E(\lambda)_K \otimes E(\mu)_L)= \emptyset$. Since $E(\lambda)_K$ is the unique simple quotient of $\mathbb{M}(\lambda)_K$ and $\mathbb{X}_{\bf T}(E(\lambda)_K)= \mathbb{X}_{\bf T}(\mathbb{M}(\lambda)_K)$, it is enough to show that
$$\op{Hom}_{\Bbbk \bf G}(\mathbb{M}(\lambda)_K \otimes \mathbb{M}(\mu)_L,  E(\theta)_J)=0.$$

Now suppose that $\op{Hom}_{\Bbbk \bf G}(\mathbb{M}(\lambda)_K \otimes \mathbb{M}(\mu)_L,  E(\theta)_J)\ne 0.$ Let $\psi$ be a nonzero homomorphism.  By Proposition \ref{MJ=KUW},  the $\Bbbk {\bf G}$-module $\mathbb{M}(\lambda)_K$ has the form $$\mathbb{M}(\lambda)_K=\sum_{w\in W^K}\Bbbk{\bf U}_{w_Kw^{-1}}\dot{w}\eta(\lambda)_K.$$
So there exists an element $x_1w_1\eta(\lambda)_K \otimes x_2w_2\eta(\mu)_L\in  \mathbb{M}(\lambda)_K \otimes \mathbb{M}(\mu)_L$ such that  $\psi(x_1w_1\eta(\lambda)_K \otimes x_2w_2\eta(\mu)_L) \ne 0$, which implies that  $\psi(x\eta(\lambda)_K \otimes w \eta(\mu)_L) \ne 0$ for some $x \in {\bf U}_{w_K}$ and $w\in W$.
 We claim  that $\psi(x\eta(\lambda)_K \otimes w \eta(\mu)_L) \in  \mathfrak{P}_K(E(\theta)_J).$ Indeed, noting that $\displaystyle E(\theta)_J=\sum_{w\in Z_J(\theta)}\Bbbk {\bf U}_{w_Jw^{-1}}\dot{w}C(\theta)_J$ by Proposition \ref{DesEJ}, we could write
 $$\xi= \psi(x\eta(\lambda)_K \otimes w \eta(\mu)_L)=\sum_{w\in Z_J(\theta), u\in{\bf U}_{w_Jw^{-1}} }a_{w,u} u\dot{w}C(\theta)_J. $$
Denote  $A=\{u\in {\bf U}_{w_Jw^{-1}} \mid a_{w,u} \ne 0\}$.   If $u_0\in A$ is not in ${\bf U}_K$, then there exists $t\in \displaystyle \bigcap_{i\in K} \op{Ker} \alpha_i$ such that $tu_0t^{-1} \notin A$. For such $t$,  we have $txt^{-1}=x$ and
thus we see that
\begin{align}\label{eq3.1}
\psi(t(x\eta(\lambda)_K \otimes w \eta(\mu)_L))= \lambda(t)\mu^w(t) \psi(x\eta(\lambda)_K \otimes w \eta(\mu)_L).
\end{align}
However, noting that $tu_0t^{-1} \notin A$, so $t\xi $  is not a scalar of  $\xi$ any more and we get a contradiction. The claim is proved. Let $M= \psi(\Bbbk {\bf U}_{w_K} \eta(\lambda)_K \otimes  \mathfrak{P}_K(\mathbb{M}(\mu)_L )$ be the $\Bbbk {\bf G}_K$-submodule of
$\mathfrak{P}_K(E(\theta)_J)$.  So we get
\begin{align}\label{eq3.101}
\op{Hom}_{\Bbbk {\bf G}_K}(\Bbbk {\bf U}_{w_K} \eta(\lambda)_K  \otimes \mathfrak{P}_K(\mathbb{M}(\mu)_L), M)\ne 0.
\end{align}
We claim that in this equation,
\begin{align}\label{eq3.2}
\mathbb{X}_{{\bf T}_K} (\Bbbk {\bf U}_{w_K} \eta(\lambda)_K  \otimes \mathfrak{P}_K(\mathbb{M}(\mu)_L) ) \cap \mathbb{X}_{{\bf T}_K}(M)=\emptyset.
\end{align}
Indeed, equation (\ref{eq3.1}) indicates that
\begin{align}\label{eq3.3}
\mathbb{X}_{{\bf T}_{I\setminus K} }(\Bbbk {\bf U}_{w_K}\eta(\lambda)_K\otimes \mathfrak{P}_K(\mathbb{M}(\mu)_L) ) = \mathbb{X}_{{\bf T}_{I\setminus K}}(M).
\end{align}
Since $\mathbb{X}_{\bf T}(E(\theta)_J) \cap  \mathbb{X}_{\bf T}(\mathbb{M}(\lambda)_K \otimes \mathbb{M}(\mu)_L)= \emptyset$
according to the initial setting, we have
\begin{align}\label{eq3.4}
\mathbb{X}_{\bf T} (\Bbbk {\bf U}_{w_K} \eta(\lambda)_K  \otimes \mathfrak{P}_K(\mathbb{M}(\mu)_L) ) \cap \mathbb{X}_{\bf T}(M)=\emptyset.
\end{align}
Combining (\ref{eq3.3}) and  (\ref{eq3.4}),  then we get (\ref{eq3.2}). It is easy to see that $\Bbbk {\bf U}_{w_K} \eta(\lambda)_K$ is the Steinberg module of $\Bbbk {\bf G}_K$.  Using  (\ref{eq3.101}) and  (\ref{eq3.2}),  we have $$\op{Hom}_{\Bbbk {\bf G}_K}(\op{St} \otimes M_1,  M_2)\ne 0,$$
where $M_1,M_2$ are simple objects in $\mathscr{X}({\bf G}_K)$ which satisfy that $\mathbb{X}_{{\bf T}_K}(\op{St} \otimes M_1 ) \cap \mathbb{X}_{{\bf T}_K}(M_2)=\emptyset$.

For the simplicity of the notation, we still consider the group ${\bf G}$ instead of ${\bf G}_K$. By the previous discussion,  we have
$$\op{Hom}_{\Bbbk {\bf G}}(\op{St} \otimes \mathbb{M}(\mu)_L,  E(\theta)_J)\ne 0,$$
 where $\mathbb{X}_{\bf T}(E(\theta)_J) \cap  \mathbb{X}_{\bf T}(\mathbb{M}(\mu)_L)= \emptyset$.
 Let $\varphi \in \op{Hom}_{\Bbbk {\bf G}}(\op{St} \otimes \mathbb{M}(\mu)_L,  E(\theta)_J)$ be a nonzero homomorphism. Thus there exists $u\in {\bf U}_{w_L}$ such that $\varphi (u\eta \otimes \eta(\mu)_L) \ne 0$. Similar to the previous discussion, we have $\varphi (u\eta \otimes \eta(\mu)_L)\in \mathfrak{P}_L(E(\theta)_J) $. Note that $\mathfrak{P}_L(\op{St})$  and $\Bbbk {\bf U}_{w_L} \eta(\mu)_L$ are  the Steinberg module of $\Bbbk {\bf G}_L$.  Thus  $\op{Hom}_{\Bbbk {\bf G}_L}(\op{St} \otimes \op{St},  N)\ne 0$ for some simple object $N\in \mathscr{X}({\bf G}_L)$, where $\op{tr} \notin  \mathbb{X}_{{\bf T}_{L}}(N)$. This is a contradiction to  Conjecture \ref{specialconjecture}.
Thus  Conjecture \ref{Homzero} is proved on the assumption  that Conjecture  \ref{specialconjecture}  holds.
\end{proof}

\section{The case ${\bf G}=SL_2(\bar{\mathbb{F}}_q)$}

In this section we consider ${\bf G}=SL_2(\bar{\mathbb{F}}_q)$ and show that  Conjecture \ref{Homfinite} and  Conjecture \ref{Homzero}  are valid. We assume that
$\text{char}\  \bar{\mathbb{F}}_q \ne 2$. Let  ${\bf T}$  be the diagonal matrices and ${\bf U}$ be the strictly upper unitriangular matrices in $SL_2(\bar{\mathbb{F}}_q)$.
 Let ${\bf B}$ be the Borel subgroup generated by ${\bf T}$ and ${\bf U}$, which is the upper triangular matrices in $SL_2(\bar{\mathbb{F}}_q)$.
As before, let ${\bf N}$ be the normalizer of ${\bf T}$ in ${\bf G}$ and  $W= {\bf N}/{\bf T}$ be the  Weyl group.  We set
$s=\begin{pmatrix}0 &1\\  -1 &0\end{pmatrix}$, which is the simple reflection of $W$.
There are two natural isomorphisms
$$h:   \bar{\mathbb{F}}^*_q\rightarrow{\bf T}, \   h(t)=\begin{pmatrix}t &0 \\0&t^{-1}\end{pmatrix}; \ \ \ \
\varepsilon : \bar{\mathbb{F}}_q\rightarrow{\bf U}, \ \  \varepsilon(x)= \begin{pmatrix}1&x\\0&1\end{pmatrix},$$
which satisfies $h(t)\varepsilon(x)h(t)^{-1}= \varepsilon(t^2x)$. The simple root
$\alpha: {\bf T}\rightarrow  \bar{\mathbb{F}}^*_q$ is given by $\alpha(h(t))= t^2$.
Moreover, one has that
\begin{align} \label{sus=xsty}
s \varepsilon(x) s=\displaystyle \varepsilon(-x^{-1}) s h(-x) \varepsilon(-x^{-1}).
\end{align}
For any $\theta\in \widehat{\bf T}$, we abbreviate $\theta(x)=\theta(h(x))$ for simplicity.

When ${\bf G}=SL_2(\bar{\mathbb{F}}_q)$, the category $\mathscr{X}({\bf G})$ has been studied in \cite{D2}. According to the results in \cite[Section 3]{D2}, all simple  objects in $\mathscr{X}({\bf G})$ are $\op{St}, \Bbbk_{\op{tr}}$ and $\mathbb{M}(\theta)$, where  $\theta \in \widehat{\bf T}$  is nontrivial.  Using Proposition \ref{conjholdforinduced}, to prove Conjecture \ref{Homfinite} and  Conjecture \ref{Homzero},  we just need to consider the simple quotients of   $\op{St} \otimes \op{St}$.  Note that  $\op{St}=\Bbbk {\bf U} \eta$, where $\eta=(1-s) {\bf 1}_{\op{tr}}$.
Using (\ref{sus=xsty}), it is easy to see that
\begin{align}\label{sueta}
s  \varepsilon(x) \eta= ( \varepsilon(-x^{-1})-1) \eta
\end{align}
for any $x\in \bar{\mathbb{F}}^*_q$.
For convenience, we denote $$\varepsilon(x) \otimes \varepsilon(y):= \varepsilon(x) \eta \otimes \varepsilon(y)\eta.$$
It is easy to see that $\{\varepsilon(x) \otimes \varepsilon(y)\mid x, y\in \bar{\mathbb{F}}_q\}$ is a basis of $\op{St} \otimes \op{St}$.

Now for any $x, y\in \bar{\mathbb{F}}_q$, we set
$$\omega_{x,y}= \frac{1}{2}(\varepsilon(x)\otimes \varepsilon(y) +\varepsilon(y) \otimes \varepsilon(x)).$$
Let $V_+$ be the space spanned  by $\{\omega_{x,y}\mid x, y\in \bar{\mathbb{F}}_q\}$.  Let
$$\rho_{x,y}=\frac{1}{2} (\varepsilon(x)\otimes \varepsilon(y) -\varepsilon(y) \otimes \varepsilon(x)),$$
where $x\ne y\in \bar{\mathbb{F}}_q$. Let $V_-$ be the space spanned by $\{\rho_{x,y}\mid x\ne y\in \bar{\mathbb{F}}_q\}$.
The following two lemmas are easily verified.

\begin{lemma} \label{omega}
For any $x, y\in \bar{\mathbb{F}}_q$, one has that
$$ h(t) \omega_{x,y}= \omega_{t^2x,t^2y}, \quad  \varepsilon(z)\omega_{x,y}=\omega_{z+x,z+y} ,$$
where $t\in \bar{\mathbb{F}}^*_q$ and $z\in \bar{\mathbb{F}}_q$. Moreover we get  $s\omega_{0,0}=\omega_{0,0}$ and
$$s\omega_{x,y}=  \omega_{-x^{-1},-y^{-1}}-\omega_{-x^{-1},0}-\omega_{-y^{-1},0}+\omega_{0,0}, \quad s\omega_{x,0}= - \omega_{-x^{-1},0}+\omega_{0,0}$$
for any  $x, y\in  \bar{\mathbb{F}}^*_q$.
\end{lemma}

\begin{proof} For any  $x, y\in  \bar{\mathbb{F}}^*_q$,  using (\ref{sueta}), we have
\begin{align}\label{sxy}
s(\varepsilon(x) \otimes \varepsilon(y))=\varepsilon(-x^{-1})\otimes \varepsilon(-y^{-1})- \varepsilon(-x^{-1})\otimes 1-1\otimes \varepsilon(-y^{-1})+ 1\otimes1.
\end{align}
By the notation of $\omega_{x,y}$,   it is easy to see that
$$s\omega_{x,y}=  \omega_{-x^{-1},-y^{-1}}-\omega_{-x^{-1},0}-\omega_{-y^{-1},0}+\omega_{0,0}.$$
The other equations are obvious and the lemma is proved.
\end{proof}

\begin{lemma} \label{rho}
For any $x, y\in \bar{\mathbb{F}}_q$ and $x\ne y$, one has that
$$ h(t) \rho_{x,y}= \rho_{t^2x,t^2y}, \quad  \varepsilon(z)\rho_{x,y}=\rho_{z+x,z+y},$$
where $t\in \bar{\mathbb{F}}^*_q$ and $z\in \bar{\mathbb{F}}_q$.  Moreover we get
$$s\rho_{x,y}=\rho_{-x^{-1},-y^{-1}}-\rho_{-x^{-1},0}+ \rho_{-y^{-1},0},   \quad  s\rho_{x,0}=-\rho_{-x^{-1}, 0}$$
for any $x, y\in  \bar{\mathbb{F}}^*_q$ and $x\ne y$.
\end{lemma}

\begin{proof}  The first two equations are easy. The last two equations follow from (\ref{sxy}) and the setting of $\rho_{x,y}$.
\end{proof}

Using these two lemmas, we see that $V_+$ and $V_-$ are  $\Bbbk {\bf G}$-submodules of $\op{St} \otimes \op{St}$.
Moreover, it is easy to see that   $\op{St} \otimes \op{St} =V_+ \oplus V_-$. The following proposition tells us that this is an indecomposable direct sum decomposition  of  $\op{St} \otimes \op{St}$.

\begin{proposition} \label{indecomposable}
One has that $V_+= \Bbbk {\bf G} \omega_{0,0}$ and  $V_-= \Bbbk {\bf B} \rho_{0, 1}$.
The $\Bbbk {\bf G}$-modules $V_+$ and $V_-$ are indecomposable.
\end{proposition}

\begin{proof}
Let $V'_+=\Bbbk {\bf G}  \omega_{0,0}$. Then $\omega_{z,z} \in V'_+$ for any $z\in \bar{\mathbb{F}}^*_q$. Using Lemma \ref{omega}, we see that
$$s\omega_{z,z}=  \omega_{-z^{-1},-z^{-1}}-2\omega_{-z^{-1},0}+\omega_{0,0}  \in V'_+,$$
and thus $\omega_{-z^{-1},0}\in V'_+$ for any $z\in \bar{\mathbb{F}}^*_q$. So we get  $V_+=V'_+= \Bbbk {\bf G}  \omega_{0,0}$.
Using Lemma \ref{rho}, for any $x, y\in \bar{\mathbb{F}}_q$ and $x\ne y$,  we get
$$\rho_{x,y}=\varepsilon(y) h((x-y)^{\frac{1}{2}})\rho_{1,0} $$
and thus $V_-= \Bbbk {\bf B} \rho_{0, 1}$. The first assertion is proved.

To prove the second assertion, it is enough to show that $\op{End}_{\Bbbk {\bf G}} (V_+)\cong \op{End}_{\Bbbk {\bf G}} (V_-)\cong \Bbbk$.
Note that $\Bbbk  \omega_{0,0}$ is the unique $\bf T$-stable line in $V_+$. It is easy to see that $\op{End}_{\Bbbk {\bf G}} (V_+)\cong  \Bbbk$ and thus $V_+$ is indecomposable.

Now we consider $\op{End}_{\Bbbk {\bf G}} (V_-)$. Let $\varphi \in \op{End}_{\Bbbk {\bf G}} (V_-)$ and we write
$$\xi= \varphi(\rho_{1,0})= \sum_{u\in \bar{\mathbb{F}}^*_q \setminus \{1\} } a_u \rho_{u,1} +\sum_{v\in \bar{\mathbb{F}}^*_q } b_v \rho_{v,0} +\sum_{x,y\in \bar{\mathbb{F}}^*_q  \setminus \{1\}}c_{x,y} \rho_{x,y},$$
where $a_u, b_v, c_{x,y}\in \Bbbk. $
Noting that $\rho_{1,0}+ h(\sqrt{-1})\varepsilon(-1) \rho_{1,0}=0$ by Lemma \ref{rho}, then
$\xi+ h(\sqrt{-1})\varepsilon(-1) \xi=0$ and thus
we get
\begin{align}\label{Eq1}
c_{x,y}+c_{1-x,1-y}=0,  \quad a_z+b_{1-z}=0
\end{align}
for any $x,y,z\in   \bar{\mathbb{F}}^*_q \setminus \{1\}$.  On the other hand, it is easy to check that $h(\sqrt{-1})s \rho_{1,0}+\rho_{1,0}=0$  by Lemma \ref{rho}. Therefore $h(\sqrt{-1})s \xi+\xi=0$  and by some easy computation,  we get
\begin{align}\label{Eq2}
a_{z}+a_{z^{-1}}=0,  \quad c_{x^{-1},y^{-1}}+c_{x,y}=0
\end{align}
for any $x,y,z\in   \bar{\mathbb{F}}^*_q \setminus \{1\}$
 and
 \begin{align}\label{Eq3}
b_{u^{-1}}-b_{u}-a_{u}+\sum_{x} c_{x,u}-\sum_{y}c_{u,y}=0
\end{align}
 for any $u\in   \bar{\mathbb{F}}^*_q \setminus \{1\}$.
Using (\ref{Eq1}) and (\ref{Eq2}), we see that
 \begin{align}\label{Eq3.01}
b_x+ b_{\frac{x}{x-1}}=0
\end{align}
 for any $x\in   \bar{\mathbb{F}}^*_q \setminus \{1\}$.

For any integer $n\in \mathbb{N}$,  we choose $\mathfrak{D}_{q^n}\subset {\bf T}$ such that $\alpha:\mathfrak{D}_{q^n}\rightarrow  \mathbb{F}^*_{q^n}$ is a bijection. Denote  $\displaystyle \zeta_n= \sum_{x\in  \mathbb{F}^*_{q^n}} \rho_{x,0}= \underline{\mathfrak{D}_{q^n}} \rho_{1,0}$. Then we have
\begin{align}\label{Eq4}
s \varepsilon(-u) \zeta_n=(q^n+1) \rho_{u^{-1},0}+ (\varepsilon(u^{-1})-1)\zeta_n
\end{align}
for any $u\in  \mathbb{F}^*_{q^n}$ by some direct computation. Noting that $\rho_{u^{-1},0}=h(u^{-\frac{1}{2}}) \rho_{1,0}$, thus
 the element $\xi= \varphi(\rho_{1,0})$ also satisfies that
 \begin{align}\label{Eq5}
s \varepsilon(-u)  \underline{\mathfrak{D}_{q^n}} \xi =(q^n+1) h(u^{-\frac{1}{2}}) \xi+ (\varepsilon(u^{-1})-1) \underline{\mathfrak{D}_{q^n}} \xi
\end{align}
for any $u\in  \mathbb{F}^*_{q^n}$ by (\ref{Eq4}).
Using (\ref{Eq2}) and  (\ref{Eq3.01}), the element $\underline{\mathfrak{D}_{q^n}} \xi $ can be written as
$$\underline{\mathfrak{D}_{q^n}} \xi =b_1 \sum_{x\in  \mathbb{F}^*_{q^n}} \rho_{x,0}+ \sum^k_{i=1} d_i \sum_{t\in \mathbb{F}^*_{q^n}} \rho_{tz_i, t}$$
for some $z_1,z_2\dots, z_k$, where $z_i \ne z_j$ and $z_iz_j\ne 1$.  We need $z_iz_j \ne 1$ since
$$ \sum_{t\in \mathbb{F}^*_{q^n}} \rho_{tz, t}=-\sum_{t\in \mathbb{F}^*_{q^n}} \rho_{tz^{-1}, t} $$
for any $z\in  \mathbb{F}^*_{q^n}$.
For each fixed $i\in \{1,2,\cdots, k\}$, we see that
 \begin{align}\label{Eq6}
s \varepsilon(-u)\sum_{t\in \mathbb{F}^*_{q^n}} \rho_{tz_i, t}=\sum_{t\ne u,  uz^{-1}_i}  \rho_{ (u-tz_i)^{-1}, (u-t)^{-1}},
\end{align}
and
 \begin{align}\label{Eq7}
 \varepsilon(u^{-1}) \sum_{t\in \mathbb{F}^*_{q^n}} \rho_{tz_i, t}= \sum_{t\ne -u^{-1},  -(uz_i)^{-1}} \rho_{tz_i+u^{-1},  t+u^{-1}}.
\end{align}
We prove that $d_i=0$ for $i=1,2,\dots,k$. Suppose that $d_1\ne 0$. We compare the elements with the form $\rho_{uz_1, u}$ on both sides of
(\ref{Eq5}), where $u\in  \mathbb{F}^*_{q^n}$.  For convenience, for any $v\in V_-$, we denote the cardinality of  the elements with the form $\ \rho_{uz_1, u}$ appearing in $v\in V_-$ by $\Xi(v)$.
Noting that
$$ s \varepsilon(-u) \rho_{x,0}= \rho_{(u-x)^{-1},u^{-1}}- \rho_{(u-x)^{-1},0}+  \rho_{u^{-1},0}$$
for any $x\ne u$ and using (\ref{Eq6}), we see that $\Xi(s \varepsilon(-u)  \underline{\mathfrak{D}_{q^n}} \xi)$ is less than the cardinality  of the set
$$\{x\in \mathbb{F}^*_{q^n} \mid 1-xu^{-1}=z^{\pm 1}_1\}\cup \bigcup_{i=1}^k \{t\in \mathbb{F}^*_{q^n} \mid (u-t)(u-tz_i)^{-1} =z^{\pm 1}_1\}$$
and thus  $\Xi(s \varepsilon(-u)  \underline{\mathfrak{D}_{q^n}} \xi) \leq 2k+2.$  Using (\ref{Eq7}) and the same discussion, we see that
$\Xi ( \varepsilon(u^{-1}) \underline{\mathfrak{D}_{q^n}} \xi ) \leq 2k+2$.  Noting that $\Xi ( \underline{\mathfrak{D}_{q^n}} \xi ) =q^n-1$ and $\Xi(\xi)$  is a fixed number determined by $\xi$, so  the cardinality of  the elements with the form $\ \rho_{uz_1, u}$ appearing on both sides of
(\ref{Eq5}) are different since the integer $n\in \mathbb{N}$ can be replaced by any integer  $m>n$ with $n|m$.
We get a contradiction  and thus  $d_i=0$ for $i=1,2,\dots,k$.
Substitute this result into (\ref{Eq5}), we see that  $a_x=0$ and $c_{x,y}=0$ for any $x,y\in \bar{\mathbb{F}}^*_q \setminus \{1\}$. Using (\ref{Eq1}),  we get $b_x=0$ for any $x\in \bar{\mathbb{F}}^*_q \setminus \{1\}$. Thus $\xi$ must be $b_1\rho_{1,0}$ for some $b_1\in \Bbbk$.
Now we get $\op{End}_{\Bbbk {\bf G}} (V_-)\cong  \Bbbk$ and thus $V_-$ is indecomposable.
\end{proof}

In the following two propositions, we consider the simple quotients of $V_+$ and $V_-$.

\begin{proposition}\label{Vplus}
  The trivial $\Bbbk {\bf G}$-module $\Bbbk_{\op{tr}}$ is the unique simple quotient of $V_+$.
\end{proposition}

\begin{proof} Let $V^0_{+}$ be the space spanned  by  the following set
 $$\{ \omega_{x,x}- \omega_{0,0}\mid x\in \bar{\mathbb{F}}^*_q\} \cup \{2 \omega_{x,y}- \omega_{0,0}\mid x \ne y\in \bar{\mathbb{F}}_q \}.$$   By Lemma \ref{omega},  it is easy to see that  $V^0_{+}$  is a $\Bbbk {\bf B}$-module.
  On the other hand, also by Lemma \ref{omega}, we have
  $$ s (\omega_{x,x}- \omega_{0,0})= (\omega_{-x^{-1},-x^{-1}}-\omega_{0,0})-(2 \omega_{-x^{-1},0}- \omega_{0,0} ) $$
$$s(2 \omega_{x,0}- \omega_{0,0}) =-(2 \omega_{-x^{-1},0}- \omega_{0,0})$$
 for any $x\in  \bar{\mathbb{F}}^*_q$ and
  $$s(2 \omega_{x,y}- \omega_{0,0})= (2 \omega_{-x^{-1},-y^{-1}}- \omega_{0,0})-(2 \omega_{-x^{-1},0}- \omega_{0,0})-(2 \omega_{-y^{-1},0}- \omega_{0,0})$$
 for any $x, y\in  \bar{\mathbb{F}}^*_q$ and $x\ne y$.
 Thus $V^0_{+}$  is a  $\Bbbk {\bf G}$-submodule of $V_{+}$.  It is easy to see that $V_{+}/V^0_{+} \cong \Bbbk_{\op{tr}}$.   In the following, we show that $V^0_{+}$ is the unique maximal submodule of  $V_{+}$.

  Let $\xi \in V_{+}\setminus V^0_{+}$. We show that $\Bbbk {\bf G} \xi=  V_{+}$. We write
  \begin{align} \label{eq4.0}
\xi=\sum_{z \in \bar{\mathbb{F}}_q} a_z \omega_{z,z}+  \sum_{x\ne y\in \bar{\mathbb{F}}_q}2 b_{x,y}\omega_{x,y} \in V_{+}.
\end{align}
Noting  that $\xi\notin V^0_{+}$,  thus we have $\displaystyle \sum_{z \in \bar{\mathbb{F}}_q} a_z +  \sum_{x\ne y\in \bar{\mathbb{F}}_q}b_{x,y} \ne 0$.
There exists an integer $n\in \mathbb{N}$, such that $\xi \in  \op{St}_{q^n} \otimes \op{St}_{q^n}$, where $\op{St}_{q^n} =\Bbbk U_{q^n} \eta$. Choose $\mathfrak{D}_{q^n}\subset {\bf T}$ such that $\alpha:\mathfrak{D}_{q^n}\rightarrow  \mathbb{F}^*_{q^n}$ is a bijection.  For $z\in  \mathbb{F}_{q^n}$, using Lemma \ref{omega},  we have
\begin{align} \label{eq4.1}
 \underline{\mathfrak{D}_{q^n}}\ \underline{U_{q^n}}\omega_{z,z} = (q^n-1)  \underline{U_{q^n}}\omega_{0,0}
\end{align}
For $x, y \in \mathbb{F}_{q^n}$ and $x\ne y$,  by Lemma \ref{omega}, we get
\begin{align} \label{eq4.2}
\underline{\mathfrak{D}_{q^n}}\ \underline{U_{q^n}}\omega_{x,y} =  (\underline{U_{q^n}} \eta)\otimes (\underline{U_{q^n}}\eta) - \underline{U_{q^n}}\omega_{0,0}.
\end{align}
Denote $a= \displaystyle  \sum_{z \in \bar{\mathbb{F}}_q} a_z$ and $b=\displaystyle   \sum_{x\ne y\in \bar{\mathbb{F}}_q}b_{x,y}  $.  For convenience, we set
$$\Lambda_n= \underline{U_{q^n}}\omega_{0,0}, \quad \Theta_n=(\underline{U_{q^n}} \eta)\otimes (\underline{U_{q^n}}\eta).$$
Combining the equations  (\ref{eq4.1}) and (\ref{eq4.2}) and the expression  (\ref{eq4.0}) of $\xi$,  it is easy to  get
\begin{align} \label{eq4.3}
\underline{\mathfrak{D}_{q^n}}\ \underline{U_{q^n}}  \xi = (a(q^n-1)-2b) \Lambda_n +2b \Theta_n.
\end{align}
By Lemma \ref{omega} and some easy computation, we see that
\begin{align} \label{eq4.4}
s\Lambda_n =\Lambda_n + (q^n-1)\omega_{0,0}- 2\sum_{z \in \mathbb{F}^*_{q^n}} \omega_{z,0},
\end{align}
and
\begin{align} \label{eq4.5}
s \Theta_n = \Theta_n-  2(q^n+1) \sum_{z \in \mathbb{F}^*_{q^n}} \omega_{z,0}+ (q^{2n}-1)\omega_{0,0}.
\end{align}
Noting that $\underline{\mathfrak{D}_{q^n}}\ \underline{U_{q^n}}\omega_{x,y} =\displaystyle  \underline{U_{q^n}} \sum_{z \in \mathbb{F}^*_{q^n}} \omega_{z,0}$,  using (\ref{eq4.2}), we get
\begin{align} \label{eq4.6}
\underline{U_{q^n}} \sum_{z \in \mathbb{F}^*_{q^n}} \omega_{z,0}=\Theta_n-  \Lambda_n.
\end{align}
Combining (\ref{eq4.4}), (\ref{eq4.5}) and (\ref{eq4.6}),  we get
\begin{equation}\label{eq4.7}
\begin{aligned}
& \underline{U_{q^n}} s\Lambda_n   = (2q^n+1) \Lambda_n-2 \Theta_n, \\
& \underline{U_{q^n}} s \Theta_n    =  (q^n+1)^2 \Lambda_n -(q^n+2) \Theta_n.
\end{aligned}
\end{equation}
Let $\underline{U_{q^n}} s$ act on both sides of  (\ref{eq4.3}) and using (\ref{eq4.7}), we have
\begin{align} \label{eq4.8}
\underline{U_{q^n}} s\underline{\mathfrak{D}_{q^n}}\ \underline{U_{q^n}}  \xi = (a(q^n-1)(2q^n+1)+2bq^{2n})\Lambda_n -(2a(q^n-1)+2bq^n)  \Theta_n.
\end{align}

Note that  the elements $\underline{\mathfrak{D}_{q^n}}\ \underline{U_{q^n}}  \xi$ and $\underline{U_{q^n}} s\underline{\mathfrak{D}_{q^n}}\ \underline{U_{q^n}}  \xi$ are linear combinations of $\Lambda_n$ and $ \Theta_n$ by (\ref{eq4.3}) and  (\ref{eq4.8}). It is not difficult to check that the
 determinant of the matrix
$$\begin{pmatrix}a(q^n-1)-2b&2b \\a(q^n-1)(2q^n+1)+2bq^{2n} & -(2a(q^n-1)+2bq^n)    \end{pmatrix}$$
is $-2(a+b)(q^n-1)((a+2b)q^n-a)$. Since $a+b\ne 0$, the matrix is invertible for sufficiently large integer $n$ and thus $\Lambda_n,  \Theta_n$ are also linear combinations of $\underline{\mathfrak{D}_{q^n}}\ \underline{U_{q^n}}  \xi$ and $\underline{U_{q^n}} s\underline{\mathfrak{D}_{q^n}}\ \underline{U_{q^n}}  \xi$.
In particular, we see that $\Lambda_n,  \Theta_n\in \Bbbk {\bf G} \xi$.
Using  $ \Lambda_n= \underline{U_{q^n}}  \omega_{0,0}\in \Bbbk {\bf G} \xi$,  we get $ \omega_{0,0}  \in \Bbbk {\bf G} \xi$ by \cite[Lemma 3.6]{CD3}. Noting that  $V_+= \Bbbk {\bf G}  \omega_{0,0}$  by Proposition \ref{indecomposable}  and thus   $\Bbbk {\bf G} \xi=  V_{+}$.
So $V^0_{+}$ is the unique maximal submodule of  $V_{+}$ and the proposition is proved.
\end{proof}

\begin{proposition}\label{Vminus}
One has that $\op{Hom}_{\Bbbk \bf G}(V_-,  M)=0$ for any simple object $M\in \mathscr{X}({\bf G})$.
\end{proposition}

\begin{proof}
For ${\bf G}=SL_2(\bar{\mathbb{F}}_q)$,  the simple  objects in $\mathscr{X}({\bf G})$ are $\op{St}, \Bbbk_{\op{tr}}$ and $\mathbb{M}(\theta)$, where  $\theta \in \widehat{\bf T}$  is nontrivial.
Firstly we show that $\op{Hom}_{\Bbbk \bf G}(V_-,  M)=0$ when $M= \Bbbk_{\op{tr}}$ or $\op{St}$.
For any $n\in \mathbb{N}$, choose $\mathfrak{D}_{q^n}\subset {\bf T}$ such that $\alpha:\mathfrak{D}_{q^n}\rightarrow  \mathbb{F}^*_{q^n}$ is a bijection. We claim that $\underline{\mathfrak{D}_{q^n}}\ \underline{U_{q^n}} \rho_{x,y}=0$ for any $x,y\in \mathbb{F}_{q^n}$.
 Note that  $\rho_{x,y}=-\rho_{y,x}$ for any $x,y\in \mathbb{F}_{q^n}$ and thus  we get
$$\underline{U_{q^n}} \rho_{-z,0}=-\underline{U_{q^n}} \rho_{0,-z}= - \underline{U_{q^n}} \rho_{z,0}$$
for any $z \in \mathbb{F}^*_{q^n}$. Therefore  we have
$$\underline{\mathfrak{D}_{q^n}}\ \underline{U_{q^n}} \rho_{x,y} =\sum_{z\in  \mathbb{F}^*_{q^n}} \underline{U_{q^n}} \rho_{z,0}=0$$
for any $x,y\in \mathbb{F}_{q^n}$.  The claim is proved. We just show
$\op{Hom}_{\Bbbk \bf G}(V_-,  \op{St})=0$ and the same argument applies to the case $M=  \Bbbk_{\op{tr}}$.
Suppose for a contradiction  that $ \op{Hom}_{\Bbbk \bf G}(V_-,  \op{St})\ne 0$ and let  $\varphi \in \op{Hom}_{\Bbbk \bf G}(V_-,  \op{St})$ be a nonzero homomorphism. Then there exists $
v\in V_-$ such that  $\varphi (v)=\eta$, where $\eta=(1-s) {\bf 1}_{\op{tr}} \in  \op{St}$.  There exists an integer $n\in \mathbb{N}$, such that $v \in  \op{St}_{q^n} \otimes \op{St}_{q^n}$, where $\op{St}_{q^n} =\Bbbk U_{q^n} \eta$.
Using the previous claim, we see that  $\underline{\mathfrak{D}_{q^n}}\ \underline{U_{q^n}} v=0$. However $\underline{\mathfrak{D}_{q^n}}\ \underline{U_{q^n}} \eta=(q^n-1) \underline{U_{q^n}} \eta$, which is nonzero. We get a contradiction and thus $\op{Hom}_{\Bbbk  \bf G}(V_-,  \op{St})=0$.

In the following,  we  show that $\op{Hom}_{\Bbbk \bf G}(V_-,  \mathbb{M}(\theta))=0$, where $\theta$ is non-trivial.
Suppose that  $\op{Hom}_{\Bbbk \bf G}(V_-,  \mathbb{M}(\theta))\ne 0$ and let $\Gamma\in \op{Hom}_{\Bbbk \bf G}(V_-,  \mathbb{M}(\theta))$ be a nonzero homomorphism.We have $V_-= \Bbbk {\bf B} \rho_{1, 0}$ by  Proposition \ref{indecomposable}. Now we write
\begin{align}\label{4.9}
\xi=\Gamma(\rho_{1,0})= a{\bf 1}_{\theta} +b s {\bf 1}_{\theta}  +\sum_{x\in \mathbb{F}^*_{q^n}} f_x \varepsilon(x) s{\bf 1}_{\theta}.
\end{align}
for some integer $n\in \mathbb{N}$. Since $h(-1)$ is in the center of $\bf G$  and $h(-1)\rho_{1,0}= \rho_{1,0}$,  we  see that $\theta(-1)=1$.

Let $z_0\in  \overline{\mathbb{F}}_{q}$ such that $z^2_0=-1$.  Firstly we assume that $\theta(z_0)=1$.
Using (\ref{sus=xsty}), we get $s \varepsilon(x) s{\bf 1}_{\theta}=\theta(-x) \varepsilon(-x^{-1}) s {\bf 1}_{\theta}$.
Then we have
\begin{align}\label{4.10}
\Gamma(h(z_0)s\rho_{1,0})=-\Gamma(\rho_{1,0}) = b {\bf 1}_{\theta} +a s {\bf 1}_{\theta}+ \sum_{x\in \overline{\mathbb{F}}^*_{q}} f_x \theta(x) \varepsilon(x^{-1}) s{\bf 1}_{\theta},
\end{align}
which implies that $a=-b$ and $f_{x^{-1}}=-\theta(x) f_x$ when considering  (\ref{4.9}).  In particular, we have $f_1=0$.  On the other hand, we have
\begin{align}\label{4.11}
\Gamma(s \varepsilon(-1)\rho_{1,0})= \Gamma(\rho_{1,0})= as {\bf 1}_{\theta}+b \varepsilon(1) s{\bf 1}_{\theta}+ \sum_{x\in \overline{\mathbb{F}}^*_{q}} f_x \theta(x-1) \varepsilon(-(x-1)^{-1}) s{\bf 1}_{\theta}.
\end{align}
Combining (\ref{4.10}) and (\ref{4.11}), we see that $a=b=0$. So (\ref{4.9}) implies that $\Gamma(\rho_{1,0})\in \Bbbk {\bf B} s{\bf 1}_{\theta} $.
Note that $V_-= \Bbbk {\bf B} \rho_{1, 0}$ by  Proposition \ref{indecomposable}.  Then we get a contradiction since $\Bbbk {\bf B} s{\bf 1}_{\theta} \ne \mathbb{M}(\theta) $.

Now we consider the case that $\theta(z_0)=-1$. Then we get
\begin{align}\label{4.12}
\Gamma(h(z_0)s\rho_{1,0})=-\Gamma(\rho_{1,0}) = -b {\bf 1}_{\theta} -a s {\bf 1}_{\theta}- \sum_{x\in \overline{\mathbb{F}}^*_{q}} f_x \theta(x) \varepsilon(x^{-1}) s{\bf 1}_{\theta},
\end{align}
Combining  (\ref{4.9}) and (\ref{4.12}), we see  that $a=b$ and $f_{x^{-1}}=\theta(x) f_x$.
For any integer $n\in \mathbb{N}$,  denote  $\displaystyle \zeta_n= \sum_{x\in  \mathbb{F}^*_{q^n}} \rho_{x,0}= \underline{\mathfrak{D}_{q^n}} \rho_{1,0}$. Then we have
\begin{align}\label{4.13}
s \varepsilon(-u) \zeta_n=(q^n+1) \rho_{u^{-1},0}+ (\varepsilon(u^{-1})-1)\zeta_n
\end{align}
for any $u\in  \mathbb{F}^*_{q^n}$ by some direct computation. Noting that $\rho_{u^{-1},0}=h(u^{-\frac{1}{2}}) \rho_{1,0}$, thus
 the element $\xi= \Gamma(\rho_{1,0})$ also satisfies that
 \begin{align}\label{4.14}
s \varepsilon(-u)  \underline{\mathfrak{D}_{q^n}} \xi =(q^n+1) h(u^{-\frac{1}{2}}) \xi+ (\varepsilon(u^{-1})-1) \underline{\mathfrak{D}_{q^n}} \xi
\end{align}
for any $u\in  \mathbb{F}^*_{q^n}$ by (\ref{4.13}). Note that
 \begin{align}\label{4.15}
\underline{\mathfrak{D}_{q^n}} \xi =\sum_{t\in  \mathfrak{D}_{q^n}} \sum_{x\in \mathbb{F}^*_{q^n}} h(t) f_x \varepsilon(x) s{\bf 1}_{\theta}= \sum_{t,x\in \mathbb{F}^*_{q^n}}  f_x \theta(t^{-\frac{1}{2}}) \varepsilon(tx) s{\bf 1}_{\theta}.
\end{align}
Thus if we write $ \underline{\mathfrak{D}_{q^n}} \xi=\displaystyle  \sum_{z\in \mathbb{F}^*_{q^n}} g_z \varepsilon(z) s{\bf 1}_{\theta}$.
Then we see that
\begin{align}\label{4.16}
 g_z= \theta(z^{-\frac{1}{2}}) \sum_{x \in \mathbb{F}^*_{q^n}} f_x \theta(x^{\frac{1}{2}}).
 \end{align}
For any fixed $u\in  \mathbb{F}^*_{q^n}$, we compute the coefficients of ${\bf 1}_{\theta}$ on the both sides of (\ref{4.14}) and then we get
$g_u=\theta(u^{-\frac{1}{2}})a(q^n+1)$. Using (\ref{4.16}), we get
$ \displaystyle \sum_{x\in \mathbb{F}^*_{q^n}} f_x \theta(x^{\frac{1}{2}})=a(q^n+1) $.   However, if we replace the integer $n$ by any another integer $m>n$ and $n|m$, we will also have $$ \sum_{x} f_x \theta(x^{\frac{1}{2}})=a(q^m+1)=a(q^n+1).$$
Thus $a=b=0$ and then $\xi=\Gamma(\rho_{1,0})\in \Bbbk {\bf B} s{\bf 1}_{\theta} $, which is a contradiction by the same discussion as before.
The proposition is proved.
\end{proof}

\begin{remark} \normalfont \label{remark}
Noting that $V_-$ is generated by $\rho_{1,0}$, so the $\Bbbk {\bf G}$-module $V_-$ has simple quotients. However, any simple quotient of $V_-$ is not in $\mathscr{X}({\bf G})$
by Proposition \ref{Vminus}. It also shows that for any two objects $M, N\in \mathscr{X}({\bf G})$, the simple quotients of $M\otimes N$ do not generally lie in $\mathscr{X}({\bf G})$.
\end{remark}

\begin{question}\normalfont \label{question}
By Proposition \ref{indecomposable}, $V_-$ is indecomposable. The simple quotients of  $V_-$ yield some new infinite-dimensional simple modules of $SL_2(\bar{\mathbb{F}}_q)$  that have not been previously observed. Is $V_-$ itself simple?
\end{question}

Using Proposition \ref{Vplus} and Proposition  \ref{Vminus}, we see that Conjecture \ref{Homfinite} and Conjecture  \ref{Homzero} hold for ${\bf G}=SL_2(\bar{\mathbb{F}}_q)$.
\begin{theorem}\label{TheoremforSL2}
Assume that ${\bf G}=SL_2(\bar{\mathbb{F}}_q)$
and let $M,N \in  \mathscr{X}({\bf G})$. One has that
$$\dim_{\Bbbk} \op{Hom}_{\Bbbk \bf G}(M\otimes N,  L)< \infty$$ for any simple object $L\in \mathscr{X}({\bf G})$. Moreover, if   $\mathbb{X}_{\bf T}(L) \cap \mathbb{X}_{\bf T}(M\otimes N)= \emptyset$, then $ \op{Hom}_{\Bbbk \bf G}(M\otimes N,  L)=0$.
\end{theorem}

For any $M, N\in \mathscr{X}({\bf G})$,
 we can define the maximal semisimple quotient of $M\otimes N$ in  $\mathscr{X}({\bf G})$ for ${\bf G}=SL_2(\bar{\mathbb{F}}_q)$ by Theorem \ref{TheoremforSL2}.
According to  the discussion in \cite[Section 5]{D2}, all the indecomposable  objects in $\mathscr{X}({\bf G})$ are $\op{St}, \Bbbk_{\op{tr}}$ and $\{\mathbb{M}(\theta)\mid \theta \in \widehat{\bf T} \}$ and each  object in $\mathscr{X}({\bf G})$ is a direct sum of these modules.
Using the following proposition,  we can explicitly determine $\mathcal{Q}(M\otimes N)$ for all  $M, N\in \mathscr{X}({\bf G})$.
\begin{proposition}
When ${\bf G}=SL_2(\bar{\mathbb{F}}_q)$,  one has that
  $$\mathcal{Q}(\op{St}\otimes \op{St})=\Bbbk_{\op{tr}} , \  \mathcal{Q}(\op{St}\otimes \mathbb{M}(\op{tr}))=\Bbbk_{\op{tr}} \oplus \op{St}, \  \mathcal{Q}(\mathbb{M}(\op{tr}) \otimes \mathbb{M}(\op{tr}))=\Bbbk_{\op{tr}} ^{\oplus 2} \oplus \op{St}$$
and $\mathcal{Q}(\op{St}\otimes \mathbb{M}(\theta))=\mathbb{M}(\theta)\oplus \mathbb{M}(\theta^s)$ for nontrivial $\theta \in \widehat{{\bf T}}$.
For any two nontrivial characters $\lambda, \mu \in \widehat{{\bf T}}$, if $\lambda \mu \ne \op{tr}$ and $\lambda^s \mu \ne \op{tr}$ (also $\lambda \mu^s \ne \op{tr}$), then
$$\mathcal{Q}(\mathbb{M}(\lambda)\otimes \mathbb{M}(\mu))=\mathbb{M}(\lambda \mu)\oplus\mathbb{M}(\lambda^s \mu) \oplus \mathbb{M}(\lambda \mu^s). $$
If $\lambda \mu =\op{tr}$, then we have
$$\mathcal{Q}(\mathbb{M}(\lambda)\otimes \mathbb{M}(\mu))=\Bbbk_{\op{tr}}  \oplus\mathbb{M}(\lambda^s \mu)
\oplus \mathbb{M}(\lambda \mu^s). $$
If $\lambda^s \mu= \op{tr}$ (also $\lambda \mu^s =\op{tr}$), then we have
$$\mathcal{Q}(\mathbb{M}(\lambda)\otimes \mathbb{M}(\mu))=\mathbb{M}(\lambda \mu)\oplus\Bbbk_{\op{tr}} \oplus \op{St}.$$
\end{proposition}

 \begin{proof} We get $\mathcal{Q}(\op{St}\otimes \op{St})=\Bbbk_{\op{tr}}$ by Proposition \ref{Vplus} and  Proposition \ref{Vminus}.
The other results  follow from Proposition \ref{tensorinducedmodule} and \cite[Theorem 3.1]{D2}.
\end{proof}

\bigskip

\noindent{\bf Acknowledgements} The author is grateful to  Nanhua Xi,  Toshiaki Shoji and Xiaoyu Chen for their suggestions and helpful discussions.

\medskip

\noindent{\bf Statements and Declarations}  The author declares that he has no conflict of interests with others.

\medskip

\noindent{\bf Data Availability}  Data sharing not applicable to this article as no datasets were generated or analysed during the current study.

\bigskip

\bibliographystyle{amsplain}

\end{document}